\documentclass[10pt]{article}
\usepackage{amssymb,mathrsfs,verbatim,graphicx,rotate,multirow}
\usepackage{amsmath,amsthm}
\usepackage{bbold}
\usepackage{graphics}
\usepackage{color}
\usepackage{float}
\usepackage{lscape}
\usepackage{setspace}
\usepackage{subfigure}
\usepackage{natbib}
\usepackage{multirow,mathrsfs}
\usepackage{enumerate}
\usepackage[mathscr]{euscript}
\usepackage{authblk}
\usepackage[hang,flushmargin]{footmisc} 

\numberwithin{equation}{section}
\newtheorem{theorem}{Theorem}

\DeclareMathOperator*{\argmax}{argmax}

\newcommand{\expit}{\textrm{expit}}

\usepackage[justification=centering]{caption}

\usepackage{array}

\newcolumntype{x}[1]{%
>{\centering\arraybackslash}p{#1}}%

\usepackage[letterpaper,left= 1.2in,right=1.2in,top=1.2in,bottom=1.2in]{geometry}
\title{Toward computerized efficient estimation in infinite-dimensional models}
\author[1]{Marco Carone}
\author[2]{Alexander R. Luedtke}
\author[3]{Mark J. van der Laan}
\affil[1]{Department of Biostatistics, University of Washington}
\affil[2]{Vaccine and Infectious Disease Division, Fred Hutchinson Cancer Research Center}
\affil[3]{Division of Biostatistics, University of California, Berkeley}

\footnotetext{\noindent Corresponding author: Marco Carone, Department of Biostatistics, University of Washington, F-644, Health Sciences Building, Box 357232, Seattle, WA 98195-7232. Email: mcarone@uw.edu.}

\allowdisplaybreaks

\bibpunct{(}{)}{;}{a}{,}{,}
\begin{document}
\maketitle

\bibliographystyle{authordate1}

\singlespacing

\begin{abstract}
Despite the risk of misspecification they are tied to, parametric models continue to be used in statistical practice because they are accessible to all. In particular, efficient estimation procedures in parametric models are simple to describe and implement. Unfortunately, the same cannot be said of semiparametric and nonparametric models. While the latter often reflect the level of available scientific knowledge more appropriately, performing efficient inference in these models is generally challenging. The efficient influence function is a key analytic object from which the construction of asymptotically efficient estimators can potentially be streamlined. However, the theoretical derivation of the efficient influence function requires specialized knowledge and is often a difficult task, even for experts. In this paper, we propose and discuss a numerical procedure for approximating the efficient influence function. The approach generalizes the simple nonparametric procedures described recently by \cite{frangakis2015biometrics} and \cite{luedtke2015biometrics} to arbitrary models. We present theoretical results to support our proposal, and also illustrate the method in the context of two examples. The proposed approach is an important step toward automating efficient estimation in general statistical models, thereby rendering the use of realistic models in statistical analyses much more accessible.
\end{abstract}

\begin{center}{\small \textbf{Keywords:} asymptotic efficiency, canonical gradient, efficient influence function, infinite-dimensional models, pathwise differentiability.}\end{center}
\doublespacing

\newpage
\section{Introduction}\label{intro}

Efficient estimation techniques are often preferred because they maximally exploit available information and minimize the uncertainty of the resulting scientific findings. Efficiency is most broadly defined in an asymptotic sense. As such, characterizing asymptotic efficiency and constructing asymptotically efficient estimators has been an important focus of methodological and theoretical research in statistics. For convenience, throughout this paper, we ascribe an asymptotic sense to the terms \emph{efficient} and \emph{efficiency}.

In the context of parametric models, a simple efficiency theory has been available for nearly a century, largely established in Fisher's work on maximum likelihood estimation. In such models, efficiency is characterized by the Cramer-Rao bounds and efficient estimators can generally be obtained via maximum likelihood (see, e.g., \citealp{hajek1970,hajek1972,lecam1972}). When parametric models are adopted in practice, it is often because they are simple and convenient to use. However, the use of such models carries the potential for model misspecification, which may have potentially serious adverse effects on the scientific process. In many scientific problems, the available background knowledge simply does not justify the use of such restrictive statistical models.

Infinite-dimensional models -- either nonparametric or semiparametric -- offer a more flexible alternative. These richer models mitigate the risk of model misspecification and more accurately reflect the level of available prior knowledge. Unfortunately, establishing efficiency bounds for target parameters in infinite-dimensional models can be a very complex task. The development of a general efficiency theory, valid for arbitrary statistical models, is a more recent accomplishment: except for the early seminal contribution of \cite{stein1956}, developments in this area began in the late 1970s and early 1980s with the works of \cite{koshevnik1977theory}, \cite{pfanzagl1982} and \cite{begun1983annals}, among others, and continued throughout the 1990s (see, e.g., \citealp{vandervaart1991annals}; \citealp{newey1994econometrica}). Notably, it builds upon notions of differential geometry and functional analysis. In certain cases, a generalized notion of maximum likelihood, as described, for example, by \cite{kiefer1956annals}, can still be used to produce efficient estimators. In other cases though, the statistical model is too complex for a maximum likelihood estimator to exist, let alone be well-behaved. This renders the pursuit of efficient estimators a substantially more difficult task in infinite-dimensional models.

A key object in this general efficiency theory is the efficient influence function, hereafter referred to as EIF. It bears this name because it is the influence function of any efficient estimator of the parameter of interest given a particular statistical model. If the EIF is known, efficiency bounds can easily be estimated, at least theoretically, and the performance of candidate estimators can be examined against an objective benchmark. Valid confidence intervals based on a given efficient estimator can also be constructed using the EIF. This is particularly useful in settings where the bootstrap is known to fail. More importantly, if the analytic form of the EIF is available, efficient estimators can be constructed rather easily. To do so, several approaches may be used, including, for example, gradient-based estimating equations (e.g., \citealp{vanderlaan2003}), Newton-Raphson one-step corrections (e.g., \citealp{pfanzagl1982}) and targeted minimum loss-based estimation (e.g.,  \citealp{vanderlaan2011}). This provides a strong motivation for deriving the EIF in a given statistical problem. Unfortunately, the analytic computation of the EIF is seldom straightforward. It generally involves finding an influence function, characterizing the tangent space of the statistical model and projecting onto it --  the effort can be mathematically intricate. Over the years, many techniques have been developed to facilitate this task in certain classes of problems -- the discretization technique of \cite{chamberlain1987jeconometrics} is one such example. Despite this, this calculation remains a rather specialized skill, mastered mostly by a small collection of theoretically-inclined researchers. The theoretical derivation of EIFs is generally not in the skill set of practicing statisticians. Yet, in many problems, it is a necessary skill to master in order to make optimal inference in more realistic statistical models. The paucity of this skill has likely constituted an impediment to a broader appreciation and adoption of semiparametric and nonparametric techniques in applications.

In view of this barrier, one naturally wonders whether a suitable numerical approximation could serve as substitute for the analytic form of the EIF, and further whether its calculation could be computerized. An affirmative answer to this question would render the implementation of efficient inferential techniques in semiparametric and nonparametric models much more accessible to practitioners,  and the impact on current statistical practice could be profound. Recently, a very important first step toward this goal was made by \cite{frangakis2015biometrics}: these authors proposed a simple numerical routine for calculating the EIF in the context of nonparametric models when the data are discrete-valued or when the parameter is a smooth functional of the distribution function. In  our discussion of their article (see \citealp{luedtke2015biometrics}), we suggested a regularization of their technique that is valid more broadly within the context of nonparametric models. Nevertheless, neither of these methods formally address the more difficult problem of computerizing the calculation of the EIF in semiparametric models. As opposed to nonparametric models, for which the tangent space is trivially described, semiparametric models generally have much more complex tangent spaces, projecting onto which may often require great skill. Identifying a numerical approach for computing the EIF in semiparametric models is therefore a more difficult but also more needed innovation. In this article, we establish and study novel representations of the EIF that naturally allows a numerical computation of the EIF of a given parameter in a given statistical model. Importantly, we do not impose constraints on the type of model that may be considered. These representations hold great promise in allowing true computerization, as we discuss below.

This paper is organized as follows. In Section 2, we present novel representations of the EIF for use in arbitrary statistical models and show how they may be used to calculate the EIF numerically. In Section 3, we establish sufficient technical conditions that guarantee the validity of these representations. We discuss various practical issues regarding the implementation of our proposal in Section 4. In Section 5, we illustrate the validity and feasibility of the approach in the context of two examples. Finally, we provide concluding remarks in Section 6. While Theorem 1 is proved in the body of the paper, the proof of Theorems 2, 3 and 4 are provided in an Appendix.

\section{Numerical calculation of the efficient influence function}\label{method}

\subsection{Preliminaries}

Suppose that we observe independent $d$-dimensional variates $X_1,X_2,\ldots,X_n$ following a distribution $P_0$ known only to belong to the statistical model $\mathscr{M}$. We denote by $\mathscr{X}(P)\subseteq\mathbb{R}^d$ the sample space associated to $P\in\mathscr{M}$. We are interested in efficiently inferring about $\psi_0:=\Psi(P_0)$ using the available data, where $\Psi:\mathscr{M}\rightarrow\mathbb{R}^q$ represents a pathwise differentiable parameter mapping of interest. Pathwise differentiability ensures the parameter is a sufficiently smooth mapping so as to admit an efficiency theory (see, e.g., \citealp{pfanzagl1982,bkrw1997}). We denote by $L_2^0(P)$ the Hilbert space of $P$-integrable functions from $\mathscr{X}(P)$ to $\mathbb{R}^q$ with mean zero and finite variance under $P$. The parameter $\Psi$ is said to be pathwise differentiable if there exists some $\chi_P\in L_2^0(P)$ such that, for each regular one-dimensional parametric submodel $\mathscr{M}_0:=\{P_\epsilon:\epsilon\in\mathscr{E}\}\subseteq \mathscr{M}$ with $\mathscr{E}\subset\mathbb{R}$ an interval containing zero and $P_{\epsilon=0}=P$, the pathwise derivative $\left.\frac{d}{d\epsilon}\Psi(P_\epsilon)\right|_{\epsilon=0}$ can be represented as the inner product $\int \chi_P(u)s(u)dP(u)$, where $s$ is the score for $\epsilon$ at $\epsilon=0$ in $\mathscr{M}_0$ \citep{pfanzagl1982}. Any such element $\chi_P$ is said to be a gradient of $\Psi$ at $P$ relative to $\mathscr{M}$. The tangent space $T_\mathscr{M}(P)$ of $\mathscr{M}$ at $P$ is defined as the closure of the linear span of scores at $P$ arising from  regular one-dimensional parametric submodels of $\mathscr{M}$ through $P$. The canonical gradient is the unique gradient contained in $T_\mathscr{M}(P)$ and corresponds to the EIF under sampling from $P$. Throughout, we will refer to the EIF at $P$ as $\phi_P$ and write $\phi_P(x)$ for the evaluation of $\phi_P$ at the observation value $x$. The asymptotic variance of an efficient estimator of $\psi_0$ relative to model $\mathscr{M}$ is given by $\int \phi_{P_0}(u) \phi_{P_0}(u)^\top dP_0(u)$. Without loss of generality, we will assume $q=1$ since the general case can be trivially dealt with using the developments herein applied to each component.

If pathwise differentiability holds uniformly over paths in a neighborhood around $P$, for any $P_1\in\mathscr{M}$ close enough to $P$,  the parameter admits the linearization \begin{align}
\Psi(P_1)-\Psi(P)\ &=\ \int \phi_{P_1}(u)d(P_1-P)(u)+R(P_1,P)\notag\\
&=\ -\int\phi_{P_1}(u)dP(u)+R(P_1,P) \label{linear}
\end{align} where $R(P_1,P)$ is a second-order remainder term, and the second line follows from the first since $\int \phi_{P_1}(u)dP_1(u)=0$ in view of the fact that the EIF is centered. This representation, which is no more than a first-order Taylor approximation over the model space, holds for most smooth parameters arising in practice. The precise form of $R$ is generally established by hand on a case-by-case basis. This linearization is critical for motivating and studying the use of both Newton-Raphson one-step correction and targeted minimum loss-based estimation to construct efficient estimators. It is also at the heart of our current proposal for obtaining a numerical approximation to the EIF value $\phi_{P}(x)$ at a given distribution $P\in\mathscr{M}$ and observation value $x\in\mathscr{X}(P)$.

\subsection{Nonparametric models}

Recently, \cite{frangakis2015biometrics} presented one such proposal based on the representation of $\phi_{P}(x)$ as the G\^ateaux derivative of $\Psi$ at $P$ in the direction of $\delta_{x}-P$, where $\delta_{x}$ represents the degenerate distribution at $x$. Of course, this can also be seen as the pathwise derivative $\left.\frac{d}{d\epsilon}\Psi(P_{\epsilon})\right|_{\epsilon=0}$ of $\Psi$ at $P$ along the linear perturbation path $\{P_{\epsilon}:=(1-\epsilon)P+\epsilon \delta_{x}:0\leq \epsilon\leq 1\}$ between $P$ and $\delta_x$ -- this simple observation will be helpful when dealing with arbitrary models. Here and throughout, any such derivative is of course interpreted as a right derivative. To computerize the process of calculating $\phi_P(x)$, these authors suggested approximating this derivative by the slope of the secant line connecting $(0,\Psi(P))$ and $(\epsilon,\Psi(P_\epsilon))$ for a very small $\epsilon>0$. In our discussion of \cite{frangakis2015biometrics} (see \citealp{luedtke2015biometrics}), we pointed out sufficient conditions that guarantee that this indeed approximates $\phi_P(x)$. For example, this approach is valid whenever the model $\mathscr{M}$ is nonparametric and the sample space $\mathscr{X}(P)$ is finite. However, if the parameter $\Psi$ depends on local features of the distribution, this method may fail when $\mathscr{X}(P)$ is infinite, such as when any component of $X$ is continuous under $P$. We proposed a slight modification of the procedure of \cite{frangakis2015biometrics} to remedy this limitation. Specifically, we proposed replacing the degenerate distribution $\delta_x$ at $x$ by a distribution $H_{x,\lambda}$ symmetric about $x$, dominated by $P$ and such that $\int g(u)dH_{x,\lambda}(u)\rightarrow g(x)$ as $\lambda\rightarrow 0$ for all $g$ in a sufficiently large class of functions. This amounts to replacing the degenerate distribution by a nearly degenerate distribution with smoothing parameter $\lambda>0$. In a technical report published contemporaneously, \cite{ichimura2015arxiv} also suggested this approach. As stated in \cite{luedtke2015biometrics}, under certain regularity conditions and provided $\mathscr{M}$ is nonparametric, it is generally the case that \begin{align}
\phi_{P}(x)\ &=\ \lim_{\lambda\rightarrow 0}\frac{d}{d\epsilon}\Psi(P_{\epsilon,\lambda})\label{npder}\\
&=\ \lim_{\lambda\rightarrow 0}\lim_{\epsilon\rightarrow 0}\frac{\Psi(P_{\epsilon,\lambda})-\Psi(P)}{\epsilon}\label{npsec}\ ,
\end{align} where we have defined the linear perturbation path $P_{\epsilon,\lambda}:=(1-\epsilon)P+\epsilon H_{x,\lambda}$. Representation \eqref{npder} is useful when the parameter is simple enough so that calculating the derivative of $\epsilon\mapsto \Psi(P_{\epsilon,\lambda})$ is analytically convenient. Otherwise, representation \eqref{npsec} can be used to circumvent this analytic step by approximating this derivative by the slope of a secant line, as in \cite{frangakis2015biometrics}. Because these representations constitute a special case of the general result described in the next subsection, we defer a statement of regularity conditions and a formal proof until then.

In practice, to approximate $\phi_P(x)$ numerically, the secant line slope exhibited in \eqref{npsec} is evaluated for small $\epsilon$ and $\lambda$. This operation only requires the ability to evaluate $\Psi$ on a given distribution. Generally, as we highlighted in \cite{luedtke2015biometrics}, $\epsilon$ must be chosen much smaller than $\lambda$ to obtain an accurate approximation -- this emphasizes that the order of the limits in \eqref{npsec} plays an important role in the implementation of this procedure. We discuss this point in greater detail later.

\subsection{Arbitrary models}\label{gen:std}

When the model is not nonparametric, the representations provided in \eqref{npder} and \eqref{npsec} generally do not hold. Except for when $\epsilon=0$, the linear perturbation path described by $P_{\epsilon,\lambda}$ is usually not contained in the model. Therefore, the parameter may not even be defined on this path. Even if it is, in general, the approximation suggested by these representations will at best yield the EIF of $\Psi$ relative to a nonparametric model rather than the actual model. While the EIF in a nonparametric model is still an influence function in $\mathscr{M}$, it is not typically efficient. This is also clear from a practical perspective: since the expressions in \eqref{npder} and $\eqref{npsec}$ do not acknowledge constraints implied by $\mathscr{M}$, they could not possibly yield the actual EIF.

Since the path $\{P_{\epsilon,\lambda}:0\leq \epsilon\leq 1\}$ is generally not in $\mathscr{M}$ for $\epsilon\neq 0$, it appears natural to consider the behavior of $\Psi$ along the analogue of the linear perturbation path in $\mathscr{M}$. To formalize this idea, we may consider the path $P^*_{\epsilon,\lambda}$ obtained by projecting $P_{\epsilon,\lambda}$ according to the Kullback-Leibler divergence into  $\mathscr{M}$ or a suitably regularized version thereof. We formally define \begin{equation}\label{proj}
P^*_{\epsilon,\lambda}:=\argmax_{P_1\in\mathscr{M}(P)}\int \log\left\{\frac{dP_1}{d\nu}(x)\right\} dP_{\epsilon,\lambda}(x)\ ,
\end{equation} where $\mathscr{M}(P):=\{P_1\in\mathscr{M}:P_1\ll P\}\subseteq \mathscr{M}$ is the subset of all probability measures in $\mathscr{M}$ that are absolutely continuous with respect to $P$, $\nu$ is a measure dominating $P$, and for any $P_1\in\mathscr{M}(P)$, $dP_1/d\nu$ is the density of $P_1$ relative to $\nu$. This projection defines a novel path in the model space. Under regularity conditions, we can then establish that  \eqref{npder} and \eqref{npsec} hold more broadly when the linear perturbation path is replaced by the model-specific path defined by this projection, as is formalized below. Here and throughout, we use the shorthand notation $\phi^*_{\epsilon,\lambda}$ to denote $\phi_{P^*_{\epsilon,\lambda}}$.

\begin{theorem}\label{representation}
Suppose that $P^*_{\epsilon,\lambda}$ exists and is in $\mathscr{M}$ for all sufficiently small $\epsilon$ and $\lambda$. Then, provided  \begin{enumerate}[\hspace{.5in}]\vspace{-.05in}
\item[\quad(A1)] (solution of EIF estimating equation) $\int \phi^*_{\epsilon,\lambda}(u)dP_{\epsilon,\lambda}(u)=0$;\vspace{-.12in}
\item[\ \ (A2)] (continuity of EIF) $\lim_{\lambda\rightarrow 0}\lim_{\epsilon\rightarrow 0}\int \phi^*_{\epsilon,\lambda}(u)d(H_{x,\lambda}-P)(u)=\phi_P(x)$;\vspace{-.12in}
\item[\ \ (A3)] (preservation of rate of convergence) $\lim_{\lambda\rightarrow 0}\lim_{\epsilon\rightarrow 0}R(P^*_{\epsilon,\lambda},P)/\epsilon=0$;
\end{enumerate} the EIF of $\Psi$ relative to $\mathscr{M}$ at $P\in\mathscr{M}$ evaluated at observation value $x$ is given by  \begin{align}
\phi_P(x)\ &=\ \lim_{\lambda\rightarrow 0}\frac{d}{d\epsilon}\Psi(P^*_{\epsilon,\lambda})\label{gender}\\
&=\ \lim_{\lambda\rightarrow 0}\lim_{\epsilon\rightarrow 0}\frac{\Psi(P^*_{\epsilon,\lambda})-\Psi(P)}{\epsilon}\ .\label{gensec}
\end{align}
\end{theorem}

\begin{proof}

Setting $P_1=P^*_{\epsilon,\lambda}$ in \eqref{linear}, we note that \begin{align*}
\Psi(P^*_{\epsilon,\lambda})-\Psi(P)\ &=\ -\int \phi^*_{\epsilon,\lambda}(u)dP(u)+R(P^*_{\epsilon,\lambda},P)\\
&=\ \int \phi^*_{\epsilon,\lambda}(u)d(P_{\epsilon,\lambda}-P)(u)-\int \phi^*_{\epsilon,\lambda}(u)dP_{\epsilon,\lambda}(u)+R(P^*_{\epsilon,\lambda},P)\ .
\end{align*} In view of (A1), we have that $
\Psi(P^*_{\epsilon,\lambda})-\Psi(P)=\int \phi^*_{\epsilon,\lambda}(u)d(P_{\epsilon,\lambda}-P)(u)+R(P^*_{\epsilon,\lambda},P)$ and since $P_{\epsilon,\lambda}-P=\epsilon(H_{x,\lambda}-P)$, we find that \begin{equation*}
\frac{\Psi(P^*_{\epsilon,\lambda})-\Psi(P)}{\epsilon}\ =\ \int \phi^*_{\epsilon,\lambda}(u)d(H_{x,\lambda}-P)(u)+\frac{R(P^*_{\epsilon,\lambda},P)}{\epsilon}\ .
\end{equation*} The result follows directly from (A2) and (A3).

\end{proof}

Condition (A1) drives in large part our generalization of the procedures of \cite{frangakis2015biometrics} and \cite{luedtke2015biometrics} to arbitratry models. Projecting the path $\{P_{\epsilon,\lambda}:0\leq\epsilon\leq 1\}$ into $\mathscr{M}$ to obtain $\{P^*_{\epsilon,\lambda}:0\leq \epsilon\leq 1\}$ is expected to ensure that the score-like equation described in (A1) is solved. In fact, as we will see in the next section, mild regularity conditions ensure that (A1) is satisfied. Condition (A2) imposes relatively weak continuity requirements on the EIF. Since $R$ is a second-order term, $R(P_{\epsilon,\lambda},P)$ is generally of order $O(\epsilon^2)$ for each fixed $\lambda>0$. Condition (A3) requires that $R(P^*_{\epsilon,\lambda},P)$ be of order $o(\epsilon)$ for $\lambda$ small and $\epsilon$ sufficiently smaller. Determining how the projection step and the smoothing parameter $\lambda>0$ affects the rate of this second-order remainder term is critical to establishing whether (A3) holds. This is studied in detail in the next section.

If the projection $P^*_{\epsilon,\lambda}$ is available in closed form, \eqref{gender} suggests that we can calculate  $\phi_P(x)$ by analytically computing the pathwise derivative of $\epsilon\mapsto\Psi(P^*_{\epsilon,\lambda})$ at $\epsilon=0$ and evaluating it at some small value of $\lambda>0$. If $P^*_{\epsilon,\lambda}$ is not available in closed form or the mapping $\epsilon\mapsto\Psi(P^*_{\epsilon,\lambda})$ is difficult to differentiate analytically, \eqref{gensec} suggests using the secant line slope \[\frac{\Psi(P^*_{\epsilon,\lambda})-\Psi(P)}{\epsilon}\] for small $\lambda$ and even smaller $\epsilon$ as an approximation to $\phi_P(x)$. Strategies for appropriately selecting values of $\epsilon$ and $\lambda$ are discussed in Section \ref{practice}, whereas the sensitivity of the approximation to these choices will be studied in the context of two examples in Section \ref{illustrations}.

Much of the effort required to use representations \eqref{gender} and \eqref{gensec} goes into identifying the projection $P^*_{\epsilon,\lambda}$ of $P_{\epsilon,\lambda}$ onto the model space. For this task, the equivalence between minimization of the Kullback-Leibler divergence and maximization of the likelihood is often useful and can be leveraged. In many cases, this projection can be identified analytically. In many others, a numerical approach must be taken. Regardless, the definition of $P^*_{\epsilon,\lambda}$ does not involve the parameter of interest. Hence, the more challenging portion of the approach is exclusively model-specific, and once it has been successfully tackled, the resulting projection can be used for any parameter a practitioner may wish to study. This contrasts sharply with the conventional approach to deriving the EIF, wherein the statistician must first derive an influence function, characterize the tangent space of the model, and finally project the influence function onto this tangent space. In this conventional approach, both the parameter-specific task -- finding an influence function -- and the model-specific task -- studying the tangent space and how to project onto it -- require specialized knowledge. Performing these tasks for a given parameter and model combination does not automatically provide an easy way of tackling any other parameter, in contrast to the approach that we propose.

\section{Verification of technical conditions}\label{theory}

The validity of representations \eqref{gender} and \eqref{gensec} is guaranteed to hold under the high-level technical conditions (A1), (A2) and (A3). We now identify lower-level sufficient conditions under which (A1), (A2) and (A3), and thus also Theorem \ref{representation}, hold.

In the developments below, we let \begin{align*}
r(\lambda)\ :=\ \left\|\frac{dH_{x,\lambda}}{dP}\right\|_{2,P}\ =\ \sqrt{\int \left\{\frac{dH_{x,\lambda}}{dP}(u)\right\}^2dP(u)}
\end{align*} denote the $L_2(P)$-norm of the Radom-Nykodim derivative of $H_{x,\lambda}$  relative to $P$. This derivative is defined for each $\lambda>0$ since $H_{x,\lambda}$ is dominated by $P$ by construction. Whenever $P$ does not assign positive mass to the set $\{x\}$, the value of $r(\lambda)$ will usually tend to infinity as $\lambda$ tends to zero. The rate at which this occurs will be critical in our study of the technical conditions listed in Theorem \ref{representation}. Here and throughout, given a function $h$, we define $\|h\|_{2,P}:=\sqrt{\int h(u)^2dP(u)}$ and $\|h\|_{\infty,A}:=\sup_{u\in A}|h(u)|$ for any set $A$. We also denote by $\mathscr{S}_{x,\lambda}$ the support of $H_{x,\lambda}$.

\subsection{Solution of the EIF estimating equation}

By virtue of being a projection, $P^*_{\epsilon,\lambda}$ is expected to solve a collection of score-like equations, including that exhibited in condition (A1). The following theorem establishes formal regularity conditions validating this heuristic argument.

\begin{theorem}\label{A1}
Condition (A1) holds provided either of the following conditions is true:\begin{enumerate}[\ \ \ (a)]
\item for some parametric submodel $\mathscr{M}_0:=\{P_\gamma:\gamma\in\Gamma\}\subseteq \mathscr{M}$ where $\Gamma\subseteq \mathbb{R}$ is an interval containing zero and $P_{\gamma=0}=P^*_{\epsilon,\lambda}$,  the function $u\mapsto \phi^*_{\epsilon,\lambda}(u)$ is the score for $\gamma$ at $\gamma=0$ in $\mathscr{M}_0$; \vspace{-.05in}
\item the Radon-Nikodym derivative of $P_{\epsilon,\lambda}$ relative to $P^*_{\epsilon,\lambda}$ is uniformly bounded in $L_2(P^*_{\epsilon,\lambda})$-norm.
\end{enumerate}
\end{theorem}

The tangent space of $\mathscr{M}$ at $P^*_{\epsilon,\lambda}$ is the collection of scores and elements that can be approximated arbitrarily well by a linear combination of scores. Under condition (a) in the above theorem, the result is established automatically since then $\phi^*_{\epsilon,\lambda}$ is itself a score. Condition (b) is a relatively milder condition. It is expected to hold in some generality since $P^*_{\epsilon,\lambda}=P_{\epsilon,\lambda}$ for $\epsilon=0$ and any $\lambda>0$, and as such, the Random-Nykodim derivative of  $P_{\epsilon,\lambda}$ relative to $P^*_{\epsilon,\lambda}$ equals one at $\epsilon=0$. Under reasonable continuity, in any small neighborhood of $\epsilon$ values near zero, this derivative is expected to be bounded in $L_2(P^*_{\epsilon,\lambda})$-norm. Additionally, any region supported by $P_{\epsilon,\lambda}$ and in which $P^*_{\epsilon,\lambda}$ assigns negligible probability mass makes a large negative contribution to the log-likelihood criterion in \eqref{proj}, thereby thwarting the objective of maximizing the likelihood. This observation further supports the plausibility of condition (b), and in fact guarantees it in the context of any finitely-supported $P_{\epsilon,\lambda}$.

\subsection{Continuity of the EIF}

We relied on certain notions of continuity to establish the validity of representations \eqref{gender} and \eqref{gensec}. The theorem below highlights how the continuity requirement stated in condition (A2) can be more concretely verified.

\begin{theorem}\label{A2}
Suppose that $\lim_{\lambda\rightarrow 0}\int \phi_P(u)dH_{x,\lambda}(u)=\phi_P(x)$ and $\lim_{\lambda\rightarrow 0}\lim_{\epsilon\rightarrow 0}\int \phi^*_{\epsilon,\lambda}(u)dP(u)=0$. Condition (A2) holds provided either \begin{center}(a) $\lim_{\lambda\rightarrow 0}\lim_{\epsilon\rightarrow 0}\|\phi^*_{\epsilon,\lambda}-\phi_P\|_{\infty,\mathscr{S}_{x,\lambda}}=0$\ \ or\ \ (b) $\lim_{\lambda\rightarrow 0}\lim_{\epsilon\rightarrow 0}r(\lambda)\|\phi^*_{\epsilon,\lambda}-\phi_P\|_{2,P}=0$.\end{center}
\end{theorem}

The requirement that $\int \phi_P(u)dH_{x,\lambda}(u)$ approximates $\phi_P(x)$ as $\lambda$ tends to zero simply stipulates that averaging $\phi_P$ with respect to a distribution eventually concentrating all its probability mass on $\{x\}$ should approximately yield $\phi_P(x)$. Furthermore, this theorem requires that $\int \phi^*_{\epsilon,\lambda}(u)dP(u)$ tends to zero, which is reasonable under some continuity since $\phi^*_{\epsilon,\lambda}$ tends to $\phi_P$ and $\int \phi_P(u)dP(u)=0$.  Beyond this, in order for condition (A2) to hold, it suffices either for  $\phi^*_{\epsilon,\lambda}$ to approximate $\phi_P$ in supremum norm over the support of $H_{x,\lambda}$ or in $L_2(P)$-norm at a rate faster than $r(\lambda)^{-1}$. These statements each hinge on a certain notion of continuity that appears needed whenever $P$ is not finitely-supported and nearly degenerate distributions must be used in defining the linear perturbation paths.

\subsection{Preservation of the rate of convergence}\label{rates}

The proof of representations \eqref{gender} and \eqref{gensec} hinges upon a linearization of the difference between $\Psi(P^*_{\epsilon,\lambda})$ and $\Psi(P)$. To ignore the remainder term from this linearization, we require that $R(P^*_{\epsilon,\lambda},P)/\epsilon$ be arbitrarily small for small enough $\lambda$ and sufficiently smaller $\epsilon$. The following theorem establishes a bound on $R(P^*_{\epsilon,\lambda},P)$ in terms of $\epsilon$ and $\lambda$ under mild conditions. It also clarifies how $\epsilon$ and $\lambda$ must be chosen to guarantee condition (A3).

\begin{theorem}\label{A3}
Suppose that there exists an interval $I_0=[m_0,m_1]\subset (0,+\infty)$ such that for each small $\lambda$ the Radon-Nikodym derivative of $P^*_{\epsilon,\lambda}$ relative to $P$ is uniformly contained in $I_0$ over the support of $P$ for sufficiently small $\epsilon$. Suppose also that there exist some $0<C<+\infty$ such that for any $P_1\in\mathscr{M}$ with Radon-Nikodym relative to $P$ bounded above by $m_1$ over the support of $P$ we have that \[|R(P_1,P)|\ \leq\  C\left\|\frac{dP_1}{dP}-1\right\|^2_{2,P}\ .\] Then, it is true that $R(P^*_{\epsilon,\lambda},P)/[\epsilon \{1+r(\lambda)\}]^2$ is bounded for small $\lambda$ and sufficiently smaller $\epsilon$. Thus, condition (A3) holds if $\epsilon=\epsilon(\lambda)$ is selected such that $\epsilon(\lambda)\{1+r(\lambda)\}^2\rightarrow 0$ as $\lambda$ tends to zero.
\end{theorem} 

As discussed in the previous subsection, since $P^*_{\epsilon,\lambda}=P$ for any value $\lambda>0$ whenever $\epsilon=0$, the derivative of $P^*_{\epsilon,\lambda}$ relative to $P$ is indeed expected to be uniformly bounded above and away from zero for small enough $\lambda$ and sufficiently smaller $\epsilon$. Furthermore, it is often the case that the remainder term, as being a second-order term arising from a linearization, can be bounded by the squared norm of the difference between the derivative of $P^*_{\epsilon,\lambda}$ relative to $P$ and its value at $\epsilon=0$. This inequality often follows quite easily from an application of the Cauchy-Schwartz inequality on the remainder term. It is easy to verify in common examples and generally holds under rather mild conditions.

\section{Practical considerations}\label{practice}

The representations presented in Theorem \ref{representation} provide the theoretical foundations for numerically approximating the EIF and thus for numerically constructing efficient estimators. The implementation of the approach suggested by these representations nevertheless presents specific challenges. Practical guidelines, as provided below, may facilitate the successful implementation of our proposal  by practitioners.

\subsection{Construction of the linear perturbation path}

In constructing the linear perturbation path that defines $P_{\epsilon,\lambda}$ and thus $P^*_{\epsilon,\lambda}$, the nearly degenerate distribution at $\{x\}$ is used instead of its purely degenerate counterpart because it ensures that all distributions along the perturbation path are dominated by $P$. This is required to ensure the validity of the representations we have proposed. Clearly, there is no need for smoothing in the components of the data unit for which the corresponding marginal distribution implied by $P$ is dominated by a counting measure. In fact, as we stress below, unnecessary smoothing will needlessly increase the computational burden of the approximation procedure. For components for which the corresponding marginal distribution is dominated by the Lebesgue measure, smoothing is generally needed. In practice, we suggest the use of product kernels for those components. Specifically, suppose that the data unit $X$ is $d$-dimensional and can be partitioned into $X=(X_{\textnormal{L}},X_{\textnormal{C}})$,  where $X_{\textnormal{L}}:=(X_{\textnormal{L}1},X_{\textnormal{L}2},\ldots,X_{\textnormal{L}d_1})$ and $X_{\textnormal{C}}:=(X_{\textnormal{C}1},X_{\textnormal{C}2},\ldots,X_{\textnormal{C}d_2})$ with $d_1+d_2=d$, and that the marginal distributions of $X_{\textnormal{L}}$ and $X_{\textnormal{C}}$ under $P$ are respectively dominated by the Lebesgue measure and a discrete counting measure. In this case, we can typically use the product kernel \[u\mapsto H_{x,\lambda}(u)\ :=\ \left[\prod_{j=1}^{d_1}K_{\lambda}(u_{\textnormal{L}j}-x_{\textnormal{L}j})\right]\times\left[\prod_{j=1}^{d_2}I(u_{\textnormal{C}j}=x_{\textnormal{C}j})\right]\ ,\] where $u:=(u_{\textnormal{L}},u_{\textnormal{C}})$ with $u_{\textnormal{L}}$ and $u_{\textnormal{C}}$ possible realizations of $X_{\textnormal{L}}$ and $X_{\textnormal{C}}$, respectively,  and $K_\lambda(w):=\lambda^{-1}K(\lambda^{-1}w)$ with $K$ some symmetric, absolutely continuous density function. The uniform kernel $K(w):=I(-1<2w<+1)$ is particularly appealing due to its simplicity, which translates to greater practical feasibility of our numerical approximation procedure. If the uniform kernel is used, it is easy to verify that $r(\lambda)=\lambda^{-d_1}$ provided, for example, $u_{\textnormal{L}}\mapsto p(u_{\textnormal{L}},x_\textnormal{C})$ is continuous and bounded away from zero in a neighborhood of $x_\textnormal{L}$. Thus, to ensure that condition (A3) is satisfied, Theorem \ref{A3} suggests choosing $\epsilon$ such that $\epsilon\ll \lambda^{2d_1}$. If $d_1$ is large, this requirement may be prohibitive, possibly even to the point of requiring a value of $\epsilon$ beyond the computer's default level of precision and thus requiring special computational techniques. Of course, while this guideline is sufficient, it may be overly conservative in some applications. In the next subsection, we provide a practical means of selecting the value of $\epsilon$ and $\lambda$.

As alluded to above, if we include smoothing over $X_{\textnormal{C}}$ as well in our choice of $H_{x,\lambda}$, we need $\epsilon\ll \lambda^{2d}$. This can be much more prohibitive computationally than requiring that $\epsilon\ll \lambda^{2d_1}$, particularly if $d_2$ is large. For this reason, smoothing in the construction of the linear perturbation path should be avoided for all components except those for which the corresponding marginal distribution under $P$ is absolutely continuous. Additionally, for some  parameters, smoothing can be avoided altogether for certain continuous components. As a general guideline for which supporting theory remains to be developed, we expect that no components require smoothing if the parameter is sufficiently smooth at $P\in\mathscr{M}$ in the sense that $\Psi(P_m)$ tends to $\Psi(P)$ for any sequence $\{P_m\in\mathscr{M}:m=1,2,\ldots\}$ for which the cumulative distribution of $P_m$ tends to that of $P$ uniformly as $m$ tends to infinity. Alternatively, if the MLE $P^*_n$ of $P$ based on observations $X_1,X_2,\ldots,X_n$ from $P$ is such that $\Psi(P_n^*)$ is a consistent estimator of $\Psi(P)$, no smoothing will generally be required. If, however, some regularization of the MLE is needed to ensure consistency (see, e.g., \citealp{vanderlaan1996annals}), smoothing will usually be critical.

\subsection{Selection of $\epsilon$ and $\lambda$ values}

When the pathwise derivative in \eqref{gender} can be calculated analytically, the approximation method proposed only involves the smoothing parameter $\lambda$. The supporting theory clearly suggests choosing $\lambda$ to be as small as possible. As we will illustrate in Section \ref{illustrations}, in some cases there is little sensitivity to the choice of $\lambda$ when \eqref{gender} is used, and even a relatively large value of $\lambda>0$ will yield stringent control of the approximation error.

Whenever the involved projection is not available in closed form or differentiation with respect to $\epsilon$ is too cumbersome to perform analytically, the secant line slope may be used to numerically approximate this analytic derivative. In such case,  $\epsilon$ and $\lambda$ must both be chosen, and more care is needed to ensure the reliability of the proposed procedure. The order of the limits in \eqref{gender} and \eqref{gensec} suggests that we must select a small value of $\lambda$ and even smaller value of $\epsilon$. This was made more precise in Section \ref{theory}, where it is prescribed to choose $\epsilon$ to be much smaller than $\lambda^{2d_1}$, where $d_1$ is the number of components of $P$ over which smoothing is required. While this theoretical requirement may serve as a rough guide in practice, it does not provide a concrete means of selecting values for $\epsilon$ and $\lambda$. For this purpose, it may be useful to produce a matrix representing the value of \[\frac{\Psi(P^*_{\epsilon,\lambda})-\Psi(P)}{\epsilon}\] as a function of $\epsilon$ and $\lambda$, both ranging over an exponential scale -- for example, we could consider both $\epsilon$ and $\lambda$ in the set $\{10^{-1},10^{-2},10^{-3},10^{-4},10^{-5},\ldots\}$. We refer to the resulting display as an epsilon-lambda plot. As a convention, the y-axis is used to represent $\epsilon$ values while $\lambda$ values are represented on the x-axis. Our theoretical findings  suggest that the right balance between $\epsilon$ and $\lambda$ will be achieved in a possibly curvilinear triangular region nested in the upper left portion of the epsilon-lambda plot. In this triangular region, the secant line slope should be essentially constant. One practical means of selecting $\epsilon$ and $\lambda$ would then consist of identifying this region visually by determining the quasi-triangular region in the upper left portion of the matrix over which the approximated EIF value is fixed up to a certain level of precision. As an illustration, without yet providing details regarding the specific parameter and model under consideration, we may scrutinize the epsilon-lambda plot arising in Example 1 from Section 5. This plot is provided as Figure 1 and clearly suggests that, up to three decimal points, the EIF value of interest is -0.963. This is indeed verified using theoretical calculations, as discussed in more detail in Section 5. The epsilon-lambda plot therefore may be a particularly useful tool for implementing the proposed approach for numerically approximating the EIF in practice.

\begin{figure}
\begin{center}
\hspace{-.2in}\includegraphics[width=5.7in]{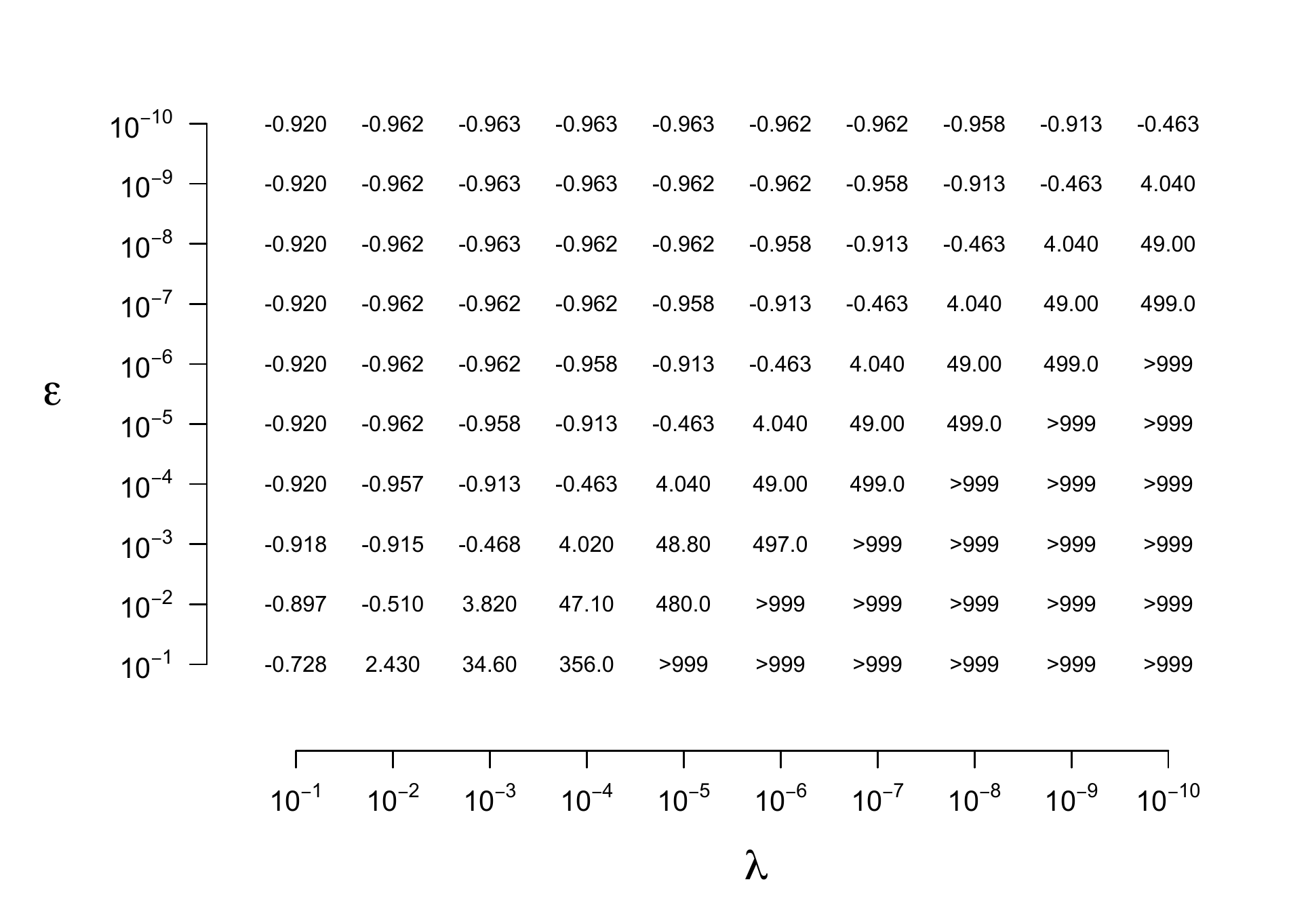}
\caption{Epsilon-lambda plot of approximated values of the EIF using a secant line slope as a function of $\epsilon$ and $\lambda$ in Example 1}
\label{elplot}
\end{center}
\end{figure}

\subsection{Numeric computation of the model space projection}

In implementing our proposal, the main challenge consists of operationalizing the optimization problem that characterizes the projection of the linear perturbation path $\{P_{\epsilon,\lambda}:0\leq \epsilon\leq 1\}$ onto the model space $\mathscr{M}$. An analytic -- or nearly analytic -- form can be found for the projection in many problems, including the illustrations provided in Section \ref{illustrations}. In other problems, the optimization problem is less analytically tractable and a numeric approach may be needed.

A general strategy for numerically approximating the required projection is to instead consider the corresponding optimization problem over $\mathscr{M}_m$, where $\mathscr{M}_1\subseteq \mathscr{M}_2\subseteq \ldots \subseteq \mathscr{M}$ is a sequence of finite-dimensional submodels of $\mathscr{M}$ such that $\cup_{m=1}^\infty \mathscr{M}_m=\mathscr{M}$.  We illustrate this in the context of families of tilted densities, though many other parametrizations are possible. We note that any distribution $Q$ dominated by $P$ can be described as a tilted form $dQ(u)=\exp\{h(u)\}dP(u)/\int \exp\{h(w)\}dP(w)$ of $P$ for some function $h\in\mathscr{H}:=\mathscr{H}(\mathscr{M})$ in a function class determined by the model $\mathscr{M}$. Here, $h$ characterizes the deviation of $Q$ from $P$. It is often easier to determine suitable approximating finite-dimensional subspaces for $\mathscr{H}$ than for $\mathscr{M}$. Suppose that $\{h_1,h_2,\ldots\}\subseteq \mathscr{H}$ forms a basis for $\mathscr{H}$, and let $\mathscr{H}_m$ denote the linear span of $\{h_1,h_2,\ldots,h_m\}$. If $P$ has density $p$ relative to $\nu$, the submodel $\mathscr{M}_m$ implied by $\mathscr{H}_m$ then consists of all distributions $Q$ with density given by \[\frac{dQ}{d\nu}(u)=\frac{\exp\big{\{}\sum_{j=1}^{m}\beta_jh_j(u)\big{\}}p(u)}{\int \exp\big{\{}\sum_{j=1}^{m}\beta_jh_j(w)\big{\}}p(w)\nu(dw)}\] for some $\underline{\beta}_m:=(\beta_1,\beta_2,\ldots,\beta_m)\in\mathbb{R}^m$. The choice $\underline{\beta}_m=(0,0,\ldots,0)$ leads to $Q=P$. Denoting by $P^*_{\epsilon,\lambda,m}$ the projection of $P_{\epsilon,\lambda}$ onto $\mathscr{M}_m$, this suggests that the corresponding optimizer $\underline{\beta}_m^*(\epsilon,\lambda)$ should be near zero for small $\epsilon$  since then $P^*_{\epsilon,\lambda}\approx P$. Thus, the search for the optimizer can be focused in a neighborhood surrounding the origin in $\mathbb{R}^m$. This simple observation can sometimes greatly accelerate the numerical optimization routine used. In practice, a sufficiently large $m$ must be selected to ensure that the resulting approximation of the projection is accurate enough to ensure the validity of the numerical evaluation of the EIF based on \eqref{gensec}. Up to an additive constant, the resulting objective function to maximize is \[\mathscr{L}(\underline{\beta}_m):=\sum_{j=1}^{m}\beta_j\int h_j(u)dP_{\epsilon,\lambda}(u)-\log\int \exp\bigg{\{}\sum_{j=1}^{m}\beta_jh_j(w)\bigg{\}}p(w)\nu(dw)\ .\] Since derivatives of $\mathscr{L}(\underline{\beta}_m)$ are easy to write down explicitly, many algorithms are available to solve this optimization problem efficiently, including Newton's method.

It may sometimes be useful to consider a stochastic version of this deterministic optimization problem. Specifically, we may generate a very large number of observations from $P_{\epsilon,\lambda}$ -- this is often easy  because $P_{\epsilon,\lambda}$ is no more than a mixture between $P$ and $H_{x,\lambda}$ -- and write the likelihood of the approximating finite-dimensional submodel based on these data. We are then faced with a standard parametric estimation problem, albeit one that may be high-dimensional. When a clever parametrization of the approximating submodel is used, it is often possible to employ standard statistical learning techniques, including regularization methods from the machine learning literature, using computationally efficient and stable off-the-shelf implementations. When adopting this approach, it appears critical to ensure that the size of the dataset generated is very large compared to the richness of the approximating submodel, since otherwise the variability resulting from this parametric estimation problem could limit our ability to achieve the required level of accuracy.

\subsection{Construction of an efficient estimator}

As emphasized earlier, knowledge of the EIF facilitates the construction of efficient estimators  in infinite-dimensional models. For example, if $\widehat{P}_n$ is a consistent estimator of $P_0\in\mathscr{M}$ based on independent draws $X_1,X_2,\ldots,X_n$ from $P_0$, the corresponding one-step Newton-Raphson estimator, defined as \begin{align*}
\psi_n^+:=\Psi(\widehat{P}_n)+\frac{1}{n}\sum_{i=1}^{n}\phi_{\widehat{P}_n}(X_i)\ ,
\end{align*} is an efficient estimator of $\psi_0$ under certain regularity conditions. The one-step approach appears to be the constructive method most amenable to an implementation based on numerical approximations of the EIF. Indeed, if the analytic form of the EIF is not known, it suffices to numerically approximate the value of $\phi_{\widehat{P}_n}(X_i)$ for each $i=1,2,\ldots,n$, rather than the entire function $u\mapsto \phi_{\widehat{P}_n}(u)$, in order to calculate $\psi_n^+$. Thus, the procedure described in this paper can be used to approximate each of these $n$ values. Nevertheless, when the projection step required to utilize the proposed representations of the EIF is computationally burdensome and the sample size $n$ is large, computing each of these values may be challenging. One need not obtain an approximation of each $\phi_{\widehat{P}_n}(X_i)$ if our objective is only to compute the one-step estimator $\psi_n^+$ -- in this case it suffices to obtain an approximation of the empirical average $\frac{1}{n}\sum_{i=1}^{n}\phi_{\widehat{P}_n}(X_i)$. This simple observation is useful because a slight modification to the representations of the EIF introduced in this paper yields a numerical procedure for approximating the required empirical average. Specifically, it is straightforward to adapt the proof of Theorem 1 to show that, under similar regularity conditions, if we define the linear perturbation $\widehat{P}_{n,\epsilon,\lambda}:=(1-\epsilon)\widehat{P}_n+\epsilon\frac{1}{n}\sum_{i=1}^{n}H_{X_i,\lambda}$ between $\widehat{P}_n$ and a uniform mixture of nearly degenerate distributions on $\{X_1\}$, $\{X_2\}$, \ldots, $\{X_n\}$, it follows that \[\frac{1}{n}\sum_{i=1}^{n}\phi_{\widehat{P}_n}(X_i)=\lim_{\lambda\rightarrow 0}\left.\frac{d}{d\epsilon}\Psi(\widehat{P}^*_{n,\epsilon,\lambda})\right|_{\epsilon=0}=\lim_{\lambda\rightarrow 0}\lim_{\epsilon\rightarrow 0}\frac{\Psi(\widehat{P}^*_{n,\epsilon,\lambda})-\Psi(\widehat{P}_n)}{\epsilon}\] with $\widehat{P}^*_{n,\epsilon,\lambda}:=\argmax_{P_1\in\mathscr{M}(P)}\int \log\left\{\frac{dP_1}{d\nu}(u)\right\}d\widehat{P}_{n,\epsilon,\lambda}(u)$. As such, a numerical approximation of the one-step estimator can be computed  in a single numerical step as \[\Psi(\widehat{P}_n)+\left.\frac{d}{d\epsilon}\Psi(\widehat{P}^*_{n,\epsilon,\lambda})\right|_{\epsilon=0}\ \approx\ \Psi(\widehat{P}_n)+\frac{\Psi(\widehat{P}^*_{n,\epsilon,\lambda})-\Psi(\widehat{P}_n)}{\epsilon}\] for appropriately selected $\epsilon$ and $\lambda$ values.

\section{Illustration and numerical studies}\label{illustrations}

To illustrate use of the representations presented above, we consider two particular examples in which the calculation of the EIF can be difficult for non-experts, whereas the approach proposed in this paper renders the problem straightforward. The technical conditions required for representations \eqref{gender} and \eqref{gensec} to hold are satisfied in these examples with the distributions selected, although we do not include details of these verifications here.

\subsection{Example 1: Average density value under known population mean}

\subsubsection{Background}

Given a distribution $P$ with Lebesgue density $p$, the average density value parameter is given by \[\Psi(P):=E_P\left\{p(X)\right\}=\int p(u)^2du\ .\] Estimation and inference for the average density value has been extensively studied in the semiparametric efficiency literature (see, e.g., \citealp{bickel1988sankhya}). We use this parameter as our first illustration becaus it is simple to describe yet requires specialized knowledge to study using conventional techniques. Suppose that $\mathscr{M}_{\textnormal{NP}}$ denotes the nonparametric model consisting of all univariate absolutely continuous distributions with finite-valued density. Suppose that $\mu\in\mathbb{R}$ is fixed and known, and denote by $\mathscr{M}\subset \mathscr{M}_{\textnormal{NP}}$ the semiparametric model consisting of all distributions in $\mathscr{M}_{\textnormal{NP}}$ with mean $\mu$. We wish to compute the EIF of $\Psi$ relative to $\mathscr{M}$ at a distribution $P\in\mathscr{M}$ evaluated at an observation value $x$.\vspace{.1in}

The EIF $\phi_{\textnormal{NP},P}$ of $\Psi$ relative to the nonparametric model $\mathscr{M}_{\textnormal{NP}}$ evaluated at $P\in\mathscr{M}$ is given by $u\mapsto \phi_{\textnormal{NP},P}(u):=2\left\{p(u)-\Psi(P)\right\}$ -- it is rather straightforward to derive this analytic form from first principles. Observing that $\mathscr{M}=\{P\in\mathscr{M}_{\textnormal{NP}}:\Theta(P)=0\}$, where $\Theta(P):=\int udP(u)-\mu$ is a pathwise differentiable parameter with EIF relative to $\mathscr{M}_{\textnormal{NP}}$ at $P\in\mathscr{M}$ given by $u\mapsto \varphi_{P}(u):=u-\mu$, Example 1 of Section 6.2 of \cite{bkrw1997} suggests that the EIF of $\Psi$ relative to $\mathscr{M}$ can be obtained as \begin{align*}
u\mapsto\phi_{P}(u)\ :=&\ \ \phi_{\textnormal{NP},P}(u)-\frac{\int \phi_{\textnormal{NP},P}(w)\varphi_P(w)dP(w)}{\int \varphi_P(w)^2dP(w)}\varphi_P(u)\\
=&\ \ 2\left\{p(u)-\Psi(P)-\frac{\int (w-\mu)p(w)dP(w)}{\int (w-\mu)^2dP(w)}(u-\mu)\right\}\ .
\end{align*} While the resulting analytic form of this EIF is relatively simple, its derivation hinges on specialized knowledge unlikely to be available to most practitioners. Use of our novel representation of the EIF provides an alternative approach that avoids the need for such knowledge, as highlighted below.

\subsubsection{Implementation and results}

To utilize our representation, we must understand how to project a given distribution $Q\in\mathscr{M}$, say with Lebesgue density $q$, into $\mathscr{M}$ relative to the Kullback-Leibler divergence. Suppose that the support of $Q$ has finite lower and upper limits $a$ and $b$, respectively, satisfying that $a<\mu<b$. An application of the method of Lagrange multipliers yields that the maximizer in $p$ of $\int \log p(u) dQ(u)$ over the class of all Lebesgue densities with mean $\mu$ is given by $q^*(u):=\{1-\xi_0(u-\mu)\}^{-1}q(u)$, where $\xi_0\in\mathbb{R}$ solves the equation \begin{equation}\int \left\{\frac{u}{ 1-(u-\mu)\xi}-\mu\right\}dQ(u)=0\end{equation} in $\xi$ and lies strictly between $(a-\mu)^{-1}$ and $(b-\mu)^{-1}$.

To compute $\phi_{P}(x)$ using the approach proposed in this paper, we must first construct the linear perturbation $P_{\epsilon,\lambda}:=(1-\epsilon)P+\epsilon H_{x\lambda}$, where $H_{x,\lambda}$ is an absolutely continuous distribution that concentrates its mass on shrinking neighborhoods of the set $\{x\}$ as $\lambda$ tends to zero. For example, we may take $H_{x,\lambda}$ to be the uniform distribution on the interval $(x-\lambda,x+\lambda)$. The projection of $P_{\epsilon,\lambda}$ onto $\mathscr{M}$ is then obtained as described in the preceding paragraph with $Q=P_{\epsilon,\lambda}$ -- as such, it has a closed-form analytic expression up to the constant $\xi_0=\xi_0(\epsilon,\lambda)$ that can be numerically solved. In the Supplementary Material, we study some properties of $\xi_0$. We may then approximate $\phi_{P}(x)$ by the secant line slope \[\phi_{P}(x)\approx \frac{\Psi(P^*_{\epsilon,\lambda})-\Psi(P)}{\epsilon}\ .\]

We evaluated this procedure numerically for a particular distribution $P$ and observation value $x$. Specifically, we took $P$ to be the Beta distribution with parameters $\alpha=3$ and $\beta=5$, and evaluated our numerical procedure for approximating the true value of $\phi_{P}(0.6)\approx -0.963$. Figure \ref{elplot_example1} provides the percent error of our numerical approximation for various combination of values for $\epsilon$ and $\lambda$. This approximation is inaccurate if either $\epsilon$ is not small enough or if $\lambda$ is too small relative to $\epsilon$. For small $\lambda$ and much smaller $\epsilon$, the secant line slope approximates the true value of $\phi_{P}(x)$ with a relative error below 0.1\%. This plot confirms what theory suggests regarding the choice of $\epsilon$ and $\lambda$. It also reaffirms the usefulness of the epsilon-lambda plot for selecting appropriate values of $\epsilon$ and $\lambda$.

\begin{figure}
\begin{center}
  \includegraphics[width=5in]{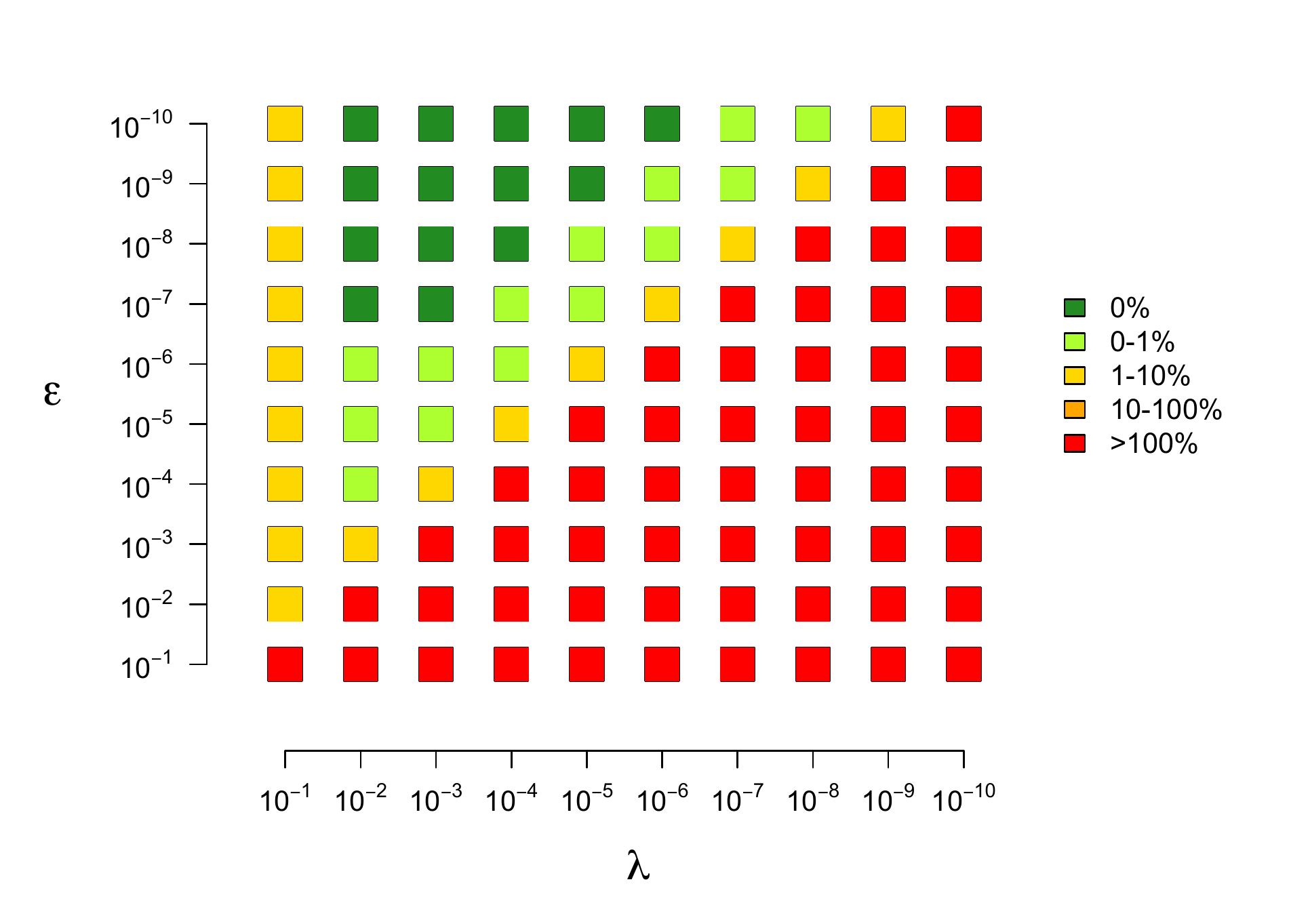}
\caption{Absolute \% error in the approximation of the EIF value using a secant line slope as a function of $\epsilon$ and $\lambda$ in Example 1}
\label{elplot_example1}
\end{center}
\end{figure}

\subsection{Example 2: G-computation parameter under Markov structure}

\subsubsection{Background}

We now consider a more complex parameter arising in the causal inference literature. Suppose that the data unit consists of the longitudinal observation $X:=(L_0,A_0,\ldots,L_{K},A_{K},L_{K+1})\sim P_0$, where $L_0,L_1,\ldots,L_{K}$ is a sequence of measurements collected at $K+1$ distinct instances through time, $L_{K+1}$ is the outcome of interest, and $A_0,A_1,\ldots,A_{K}$ are intervention indicators corresponding to each pre-outcome timepoint. For simplicity, we consider all treatment indicators to be binary. Let $\mathscr{M}_{\textnormal{NP}}$ be a nonparametric model. In practice, we may be interested in the covariate-adjusted, treatment-specific mean $\psi_0:=\Psi(P_0)$ corresponding to the intervention $(A_0,A_1,\ldots,A_K)=(1,1,\ldots,1)$. Here, for any given $P\in\mathscr{M}_{\textnormal{NP}}$, the parameter value $\Psi(P)$ is defined explicitly as $E_{P}\left[m_{0,P}(L_0)\right]$ via the G-computation recursion\begin{align*}
m_{j,P}(\overline{\ell}_j):=E_P\left[\ m_{j+1,P}(\overline{L}_{j+1})\ \middle|\ \overline{L}_{j}=\overline{\ell}_j,A_{j}=A_{j-1}=\ldots=A_0=1\ \right]
\end{align*} for $j=K,K-1,\ldots,0$, where we have set $m_{K+1,P}(\overline{L}_{K+1}):=L_{K+1}$ \citep{robins1986}. Here, for any vector $u:=(u_0,u_1,\ldots)$ we write $\overline{u}_k:=(u_0,u_1,\ldots,u_k)$. This parameter only depends on $P$ through the conditional distribution $P_{j,1}$ of $\overline{L}_{j+1}$ given $\overline{L}_{j}$ and $A_0=A_1=\ldots=A_{j}=1$ for $j=0,1,\ldots,K$, and the marginal distribution $P_{0,1}$ of $L_0$. Under certain untestable causal assumptions, $\psi_0$ corresponds to the mean of the counterfactual outcome $Y$ defined by an intervention setting all treatment nodes to one. With respect to $\mathscr{M}_{\textnormal{NP}}$, or any model with restrictions only on the conditional distribution of $A_{j}$ given $\overline{A}_{j-1}$ and $\overline{L}_{j}$ possibly for any $j\in\{0,1,\ldots,K\}$, the EIF of $\Psi$ at $P$ is known to be given by $\phi_{\textnormal{NP},P}:=\sum_{j=0}^{K+1}\phi_{j,NP,P}$, where  $\phi_{0,NP,P}(x):=m_{0,P}(\ell_0)-\Psi(P)$ and \[\phi_{j,NP,P}(x):=\frac{a_0a_1\cdots a_{j-1}}{\prod_{r=0}^{j-1}P(A_r=1\mid \overline{L}_r=\overline{\ell}_r,A_{0}=A_1=\ldots=A_{r-1}=1)}\left\{m_{j,P}(\overline{\ell}_{j})-m_{j-1,P}(\overline{\ell}_{j-1})\right\}\] for $j=1,2,\ldots,K+1$.

Let the model $\mathscr{M}$ consist of the subset of distributions $P$ in $\mathscr{M}_{\textnormal{NP}}$ such that,  for each $j=2,3,\ldots,K+1$, $L_{j}$ and $\overline{L}_{j-2}$ are independent given $L_{j-1}$ and $A_{j-1}=A_{j-2}=\ldots=A_0=1$ under $P$. For each $P\in\mathscr{M}$, we note that $m_{j,P}(\overline{\ell}_j)=m_{j,P}(\ell_j)$ for each $j$. The EIF of $\Psi$ relative to $\mathscr{M}$ at $P$  is given by $\phi_{P}:=\sum_{j=0}^{K+1}\phi_{j,P}$, where $\phi_{0,P}=\phi_{0,NP,P}$ and $\phi_{j,P}$ is defined pointwise as \begin{align*}
x\mapsto \phi_{j,P}(x)\ :=&\ \ E_P\left[\phi_{j,NP,P}(X)\mid L_{j}=\ell_{j},L_{j-1}=\ell_{j-1},\overline{A}_{j-1}=\overline{a}_{j-1}\right]\\
&\hspace{1.5in}-E_P\left[\phi_{j,NP,P}(X)\mid L_{j-1}=\ell_{j-1},\overline{A}_{j-1}=\overline{a}_{j-1}\right]\\
=&\ \ a_0a_1\cdots a_{j-1}\cdot T_j(P)(x)\cdot \{m_{j,P}(\ell_{j})-m_{j-1,P}(\ell_{j-1})\}
\end{align*} for $j=1,2,\ldots,K+1$, and we use $T_j(P)(x)$ to denote \[E_P\left[\frac{1}{\prod_{r=0}^{j-1}P(A_r=1\mid \overline{L}_r,A_{0}=A_1=\ldots=A_{r-1}=1)}\ \middle|\ L_{j}=\ell_{j},L_{j-1}=\ell_{j-1},\overline{A}_{j-1}=\overline{a}_{j-1}\right].\] Deriving this expression requires specialized knowledge and familiarity with efficiency theory for longitudinal structures. Furthermore, even given this analytic expression, the EIF may often be difficult to compute since it involves rather elaborate conditional expectations. 

\subsubsection{Implementation and results}

As in the previous example, the main challenge is to understand how to project a given distribution $Q$ into $\mathscr{M}$. Given a dominating measure $\nu$, we denote the density function of $Q$ with respect to $\nu$ as $q$. Furthermore, we denote by $q_{L_j}$ the density of the conditional distribution of $L_j$ given $\overline{L}_{j-1}$ and $\overline{A}_{j-1}$, and by $q_{A_j}$ the density of the conditional distribution of $A_j$ given $\overline{L}_{j}$ and $\overline{A}_{j-1}$. We also denote by $q_{L_j,1}$ the density $q_{L_j}$ with $\bar{a}_{j-1}=(1,1,\ldots,1)$. We use the same notational convention for any other candidate density $p$. Because for any candidate $p$ we can write \[\int \log p(u)dQ(u)=\sum_{j=0}^{K+1}\int \log p_{L_j}(\ell_j\mid \overline{\ell}_{j-1},\overline{a}_{j-1})dQ(u)+\sum_{j=0}^{K}\int \log p_{A_j}(a_j\mid \overline{\ell}_{j},\overline{a}_{j-1})dQ(u)\] and $\mathscr{M}$ can be written as a product model for the set of conditional distributions implied by the joint distribution, the required optimization problem can be performed separately for each conditional density. Because computing $\Psi(Q^*)$ does not require any component of $Q^*$ beyond $q^*_{L_j,1}$ for $j=0,1,\ldots,K+1$, we focus our attention on the corresponding optimization problems alone. Below, we denote by $\overline{q}(\ell_j,\ell_{j-1})$ the marginalized density $\iint\cdots\int q(\ell_0,1,\ell_1,1,\ldots,\ell_{j-1},1,\ell_j)\nu(d\ell_0,d\ell_1,\ldots,d\ell_{j-2})$. To find $q^*_{L_j,1}$ for $j=2,3,\ldots,K+1$, we must maximize the criterion \begin{align*}
L(p_{L_j,1})\ :=&\ \iint\cdots\int \log p_{L_j,1}(\ell_j\mid \ell_{j-1})q(\ell_0,1,\ell_1,1,\ldots,\ell_{j-1},1,\ell_j)\nu(d\ell_0,d\ell_1,\ldots,d\ell_j)\\
=&\ \iint \log p_{L_j,1}(\ell_j\mid \ell_{j-1})\overline{q}(\ell_{j-1},\ell_{j})\nu(d\ell_{j-1},d\ell_j)\\
=&\ \iint \log p_{L_j,1}(\ell_j\mid \ell_{j-1})\frac{\overline{q}(\ell_{j-1},\ell_{j})}{\int \overline{q}(\ell_{j-1},\ell_{j})\nu(d\ell_{j})}\nu(d\ell_j) \int \overline{q}(\ell_{j-1},\ell_{j})\nu(d\ell_j)\nu(d\ell_{j-1})
\end{align*} over the class of candidate conditional densities that do not depend on $\overline{\ell}_{j-2}$, here represented by $p_{L_j,1}$. Since for each fixed $\ell_{j-1}$ the mapping $\ell_j\mapsto \overline{q}(\ell_{j-1},\ell_{j})/\int \overline{q}(\ell_{j-1},\ell'_{j})\nu(d\ell'_j)$ defines a proper conditional density, by Jensen's inequality, $L(p_{L_j,1})$ is maximized by \[q^*_{L_j,1}(\ell_j\mid \ell_{j-1})=\frac{\overline{q}(\ell_{j-1},\ell_{j})}{\int \overline{q}(\ell_{j-1},\ell'_{j})\nu(d\ell'_j)}\ .\] It is easy to see that $\mathscr{M}$ constrains  neither $p_{L_1,1}$ nor $p_{L_0}$ and therefore $q^*_{L_1,1}=p_{L_1,1}$ and $q^*_{L_0}=p_{L_0}$. Thus, in the context of a longitudinal data structure, the projection of any given distribution $Q$ into a model only constrained by a Markov structure has an analytic closed-form.

As before, to compute $\phi_{P}(x)$ using the proposed representations of the EIF, we first construct the linear perturbation $P_{\epsilon,\lambda}:=(1-\epsilon)P+\epsilon H_{x,\lambda}$, where $H_{x,\lambda}$ is a distribution dominated by $P$ and concentrating its mass in shrinking neighborhoods of the set $\{x\}$ as $\lambda$ tends to zero. The projection $P^*_{\epsilon,\lambda}$ of $P_{\epsilon,\lambda}$ onto $\mathscr{M}$ has an explicit form given in the preceding paragraph with $Q=P_{\epsilon,\lambda}$. As in Example 1, we may approximate $\phi_{P}(x)$ by the secant line slope $\{\Psi(P^*_{\epsilon,\lambda})-\Psi(P)\}/\epsilon$ for small $\lambda$ and even smaller $\epsilon$. Because in this example  $P^*_{\epsilon,\lambda}$ is available in closed form, $\phi_{P}(x)$ can alternatively be approximated by $\left.\frac{d}{d\epsilon}\Psi(P^*_{\epsilon,\lambda})\right|_{\epsilon=0}$ for small $\lambda$.

For convenience, in our numerical evaluation of the EIF, we restricted our attention to a setting with $K=2$ post-baseline time-points. We considered the joint distribution $P$ of $X$ defined in terms of the following conditional distributions. The baseline covariate $L_0$ has a discrete uniform distribution on the set $\{0,1,2,3,4\}$. Given $L_0=\ell_0$, $A_0$  has a Bernoulli distribution with success probability $\expit(-1+0.5\ell_0)$. Given $A_0=a_0$ and $L_0=\ell_0$, $L_1$ has a normal distribution with mean $3\ell_0-3a_0$ and variance 4. Given $L_1=\ell_1$, $A_0=a_0$ and $L_0=\ell_0$, $A_1$ has a Bernoulli distribution with success probability $\expit\{-5+c_{10}(\ell_1)+a_0+0.5\ell_0\}$, where we define $c_{10}$ to be  the trimming function $u\mapsto -10\cdot I_{(-\infty,-10)}(u)+u\cdot I_{[-10,+10]}(u)+10\cdot I_{(+10,+\infty)}(u)$. Given $A_1=a_1$, $L_1=\ell_1$, $A_0=a_0$ and $L_0=\ell_0$, $Y$ has a Bernoulli distribution with success probability $\expit\{-1+0.5c_{10}(\ell_1)-0.5a_1-a_0\}$. We evaluated the approximations of $\phi_{P}(x)$ based on either the secant line slope or the analytic pathwise derivative at various possible values of the realized data unit  $x$. We report the absolute percent error for observation value $x:=(0,1,2,1,1)$ using the secant line slope approach in Figure \ref{elplot_example2_1color} and using the analytic derivative approach in Figure \ref{elplot_example2_2color}. The pattern observed in Figure \ref{elplot_example2_1color} is similar to that seen in Figure \ref{elplot_example1}. In a triangular region contained in the upper left portion of the epsilon-lambda plot, the approximation provided by the secant line slope is very accurate. Outside of this region, that is, for inappropriate choices of $\epsilon$ and $\lambda$, the approximation can be poor. Thankfully, the epsilon-lambda plot provides an easy way of identifying these appropriate values. From Figure \ref{elplot_example2_2color}, we note that a high level of accuracy is achieved with a relatively large $\lambda>0$. Thus, use of the analytic derivative essentially eliminates the careful selection of approximation parameters otherwise needed. Results for other observation values examined yielded similar patterns and are therefore not reported here.

\begin{figure}
\begin{center}
  \includegraphics[width=5in]{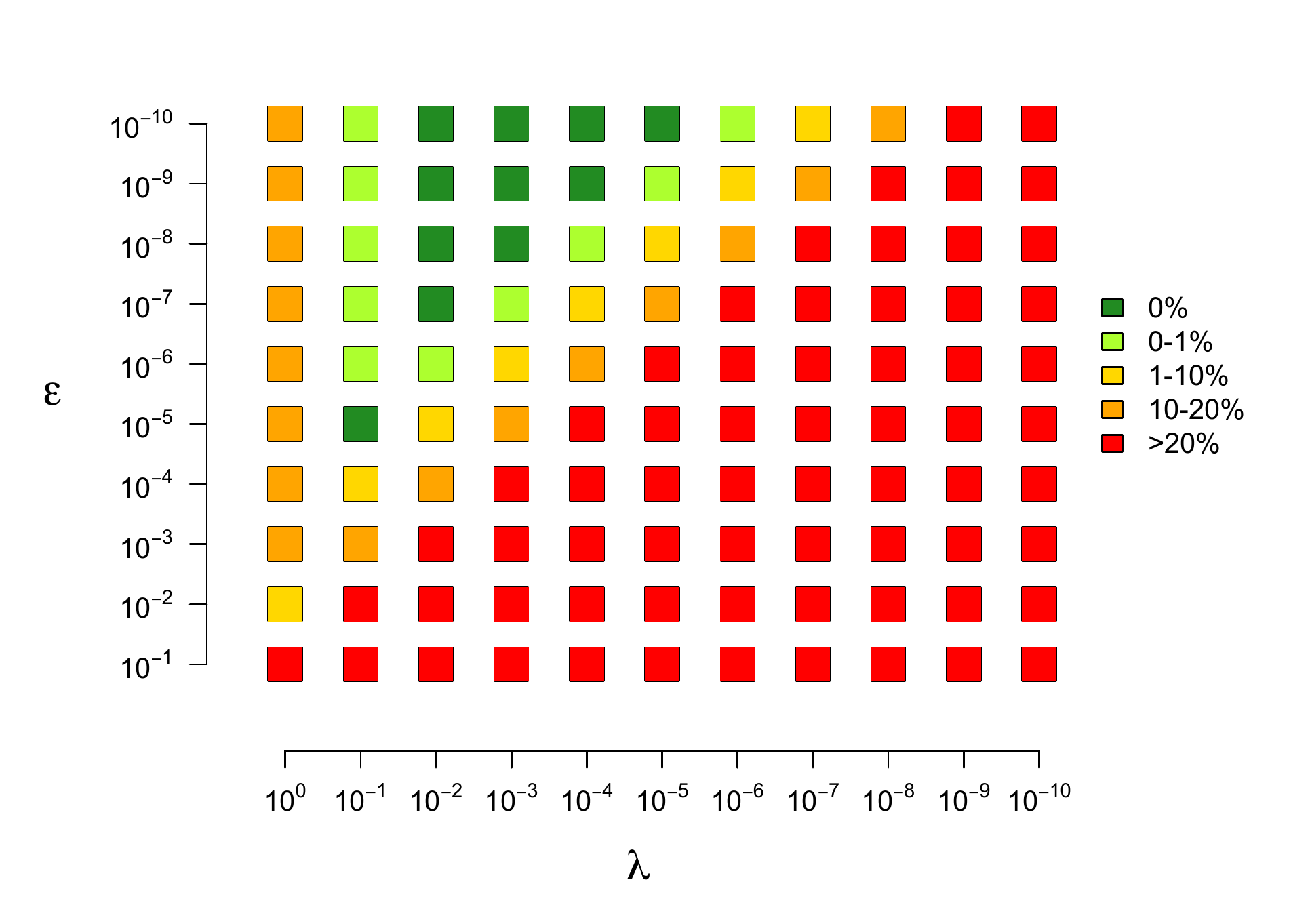}
\caption{Absolute \% error in the approximation of the EIF value using a secant line slope as a function of $\epsilon$ and $\lambda$ in Example 2}
\label{elplot_example2_1color}
\end{center}
\end{figure}

\begin{figure}
\begin{center}
  \includegraphics[width=4.75in]{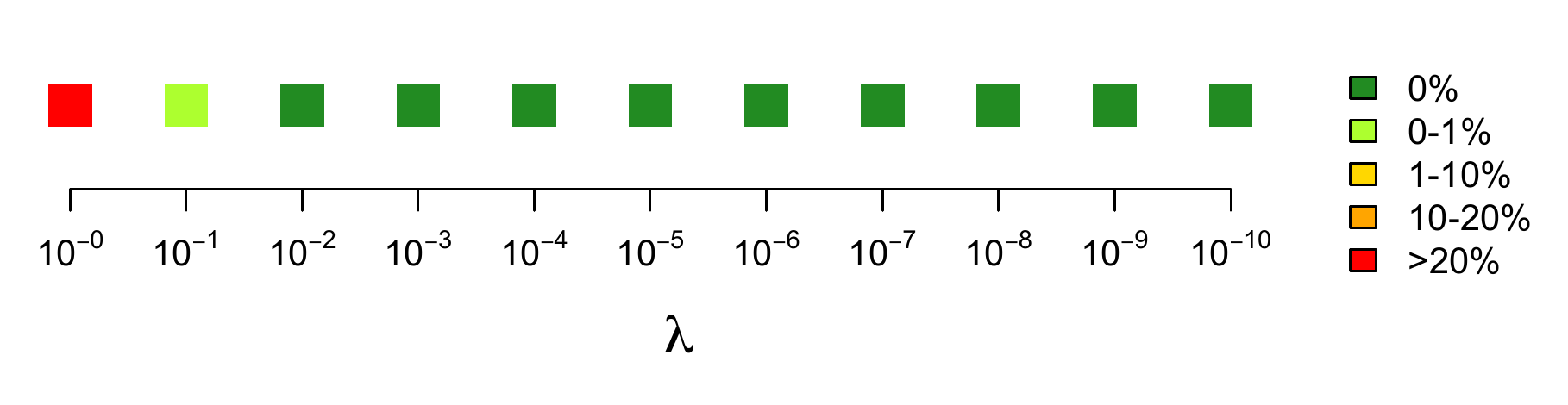}
\caption{Absolute \% error in the approximation of the EIF value using an analytic derivative as a function of $\epsilon$ and $\lambda$ in Example 1}
\label{elplot_example2_2color}
\end{center}
\end{figure}

\section{Concluding remarks}\label{conclusion}

The representations of the EIF we have presented in this paper suggest a natural strategy for numerically approximating the EIF. These representations hold in arbitrary models under mild regularity conditions. Use of these representations  requires the ability to project a given distribution into the statistical mode -- this is essentially no more than a maximum likelihood step that can be tackled by most practitioners. Most importantly, the involved work requires neither knowledge of efficiency theory nor familiarity with concepts from functional analysis or differential geometry. As such, these representations have the potential of democratizing the calculation of the EIF and thus the construction of efficient estimators in nonparametric and semiparametric models. Even for seasoned researchers in semiparametric and nonparametric theory, they provide an alternate means of tackling difficult problems, including those for which the EIF is either difficult or impossible to derive analytically.

In most problems, we anticipate the analytic work required to obtain the projection of the linear perturbation path onto the model space to be much simpler than that needed for the conventional tangent space approach. Nevertheless, this may still constitute a barrier for some practitioners. However, because the task of projecting onto the model space represents no more than an optimization problem, albeit an infinite-dimensional one, off-the-shelf computational tools may readily be used to circumvent most, if not all, analytic work otherwise required. This is particularly encouraging since strong computational skills are commonplace in statistics and data science. Furthermore, the numerical challenge will become increasingly surmountable as the capability of our computational devices continues to grow over time. It may therefore be particularly fruitful to invest additional energy into devising and studying broad numerical strategies for computerizing the calculation of the EIF based on the representations in this paper.

As with all methods that incorporate some level of automation and more readily lend themselves to use by non-specialists, there is a clear potential for misuse of the results we have presented. This appears to be an inevitable risk inherent to this type of proposal, and it equally applies to some of the most celebrated tools in current statistical practice, including the bootstrap. Deriving the EIF analytically undoubtedly remains the gold-standard approach and it should be preferred whenever possible since much information can be learned about the problem at hand from the analytic form of the EIF. In particular, verification of the regularity conditions invoked in this paper can be difficult without prior analytic knowledge of the EIF. Nevertheless,  the representations introduced in this paper have the potential of serving as an important new tool  in the arsenal of statistical researchers and practitioners alike for performing semiparametric and nonparametric analyses. Devising algorithms for verifying the required regularity conditions in any given problem is an important avenue for future research.

We have noted that a distinct advantage of the representations we have provided is that once they have been used to compute the EIF of a certain parameter in a given statistical model, the EIF of any other parameter can be obtained without any additional work since the bulk of the work required is exclusively model-specific. Nevertheless, the involved computational work must be repeated for each observation value at which we wish to evaluate the EIF. In particular, this makes it difficult to approximate the entire EIF as a function, particularly in the case of continuous or longitudinal data units. While the one-step approach only requires the EIF at the observed data points, the implementation of other efficient estimators with potentially better properties, such as targeted minimum loss-based estimators (TMLE), generally requires the entire EIF. The representations presented in this paper are therefore not conducive to a computerized implementation of TMLE. There is promise that alternative representations may be better suited for this purpose -- this is an area of active research.

\singlespacing
{\footnotesize
\vspace{.05in}
\section*{Acknowledgments} MC gratefully acknowledges the support of NIAID grant 5UM1AI068635 and the Career Development Fund of the Department of Biostatistics at the University of Washington. MvdL gratefully acknowledges the support of NIAID grant 5R01AI074345.
\vspace{.05in}
\bibliography{computerization}

\begin{thebibliography}{21}
\providecommand{\natexlab}[1]{#1}
\providecommand{\url}[1]{\texttt{#1}}
\expandafter\ifx\csname urlstyle\endcsname\relax
  \providecommand{\doi}[1]{doi: #1}\else
  \providecommand{\doi}{doi: \begingroup \urlstyle{rm}\Url}\fi

\bibitem[Begun et~al.(1983)Begun, Hall, Huang, and Wellner]{begun1983annals}
J.M. Begun, W.J. Hall, W.M. Huang, and J.A. Wellner.
\newblock Information and asymptotic efficiency in parametric-nonparametric
  models.
\newblock \emph{The Annals of Statistics}, pages 432--452, 1983.

\bibitem[Bickel and Ritov(1988)]{bickel1988sankhya}
P.J. Bickel and Y.~Ritov.
\newblock Estimating integrated squared density derivatives: sharp best order
  of convergence estimates.
\newblock \emph{Sankhy{\=a}: The Indian Journal of Statistics, Series A}, pages
  381--393, 1988.

\bibitem[Bickel et~al.(1997)Bickel, Klaassen, Ritov, and Wellner]{bkrw1997}
P.J. Bickel, C.A.J. Klaassen, Y.~Ritov, and J.A. Wellner.
\newblock \emph{Efficient and adaptive estimation for semiparametric models}.
\newblock Springer, 1997.

\bibitem[Chamberlain(1987)]{chamberlain1987jeconometrics}
G.~Chamberlain.
\newblock Asymptotic efficiency in estimation with conditional moment
  restrictions.
\newblock \emph{Journal of Econometrics}, 34\penalty0 (3):\penalty0 305--334,
  1987.

\bibitem[Frangakis et~al.(2015)Frangakis, Qian, Wu, and
  D{\'\i}az]{frangakis2015biometrics}
C.E. Frangakis, T.~Qian, Z.~Wu, and I.~D{\'\i}az.
\newblock Deductive derivation and {T}uring-computerization of semiparametric
  efficient estimation (with discussion).
\newblock \emph{Biometrics}, 2015.

\bibitem[H{\'a}jek(1970)]{hajek1970}
J.~H{\'a}jek.
\newblock A characterization of limiting distributions of regular estimates.
\newblock \emph{Zeitschrift f{\"u}r Wahrscheinlichkeitstheorie und verwandte
  Gebiete}, 14\penalty0 (4):\penalty0 323--330, 1970.

\bibitem[H{\'a}jek(1972)]{hajek1972}
J.~H{\'a}jek.
\newblock Local asymptotic minimax and admissibility in estimation.
\newblock In \emph{Proceedings of the sixth Berkeley symposium on mathematical
  statistics and probability}, volume~1, pages 175--194, 1972.

\bibitem[Ichimura and Newey(2015)]{ichimura2015arxiv}
H.~Ichimura and W.K. Newey.
\newblock The influence function of semiparametric estimators.
\newblock \emph{arXiv preprint arXiv:1508.01378}, 2015.

\bibitem[Kiefer and Wolfowitz(1956)]{kiefer1956annals}
J.~Kiefer and J.~Wolfowitz.
\newblock Consistency of the maximum likelihood estimator in the presence of
  infinitely many incidental parameters.
\newblock \emph{The Annals of Mathematical Statistics}, pages 887--906, 1956.

\bibitem[Koshevnik and Levit(1977)]{koshevnik1977theory}
Y.A. Koshevnik and B.Y. Levit.
\newblock On a non-parametric analogue of the information matrix.
\newblock \emph{Theory of Probability \& Its Applications}, 21\penalty0
  (4):\penalty0 738--753, 1977.

\bibitem[Le~Cam(1972)]{lecam1972}
L.~Le~Cam.
\newblock Limits of experiments.
\newblock In \emph{Proceedings of the sixth Berkeley symposium on mathematical
  statistics and probability}, volume~1, pages 245--261, 1972.

\bibitem[Luedtke et~al.(2015)Luedtke, Carone, and van~der
  Laan]{luedtke2015biometrics}
A.R. Luedtke, M.~Carone, and M.J. van~der Laan.
\newblock A discussion of ``{D}eductive derivation and {T}uring-computerization
  of semiparametric efficient estimation'' by {F}rangakis et al.
\newblock \emph{Biometrics}, 2015.

\bibitem[Newey(1994)]{newey1994econometrica}
W.K. Newey.
\newblock The asymptotic variance of semiparametric estimators.
\newblock \emph{Econometrica}, pages 1349--1382, 1994.

\bibitem[Pfanzagl(1982)]{pfanzagl1982}
J.~Pfanzagl.
\newblock \emph{Contributions to a general asymptotic statistical theory}.
\newblock Springer, 1982.

\bibitem[Robins(1986)]{robins1986}
J~M Robins.
\newblock {A new approach to causal inference in mortality studies with a
  sustained exposure period—application to control of the healthy worker
  survivor effect}.
\newblock \emph{Mathematical Modelling}, 7\penalty0 (9):\penalty0 1393--1512,
  1986.

\bibitem[Stein(1956)]{stein1956}
C.~Stein.
\newblock Efficient nonparametric testing and estimation.
\newblock In \emph{Proceedings of the third Berkeley symposium on mathematical
  statistics and probability}, volume~1, pages 187--195, 1956.

\bibitem[van~der Laan(1996)]{vanderlaan1996annals}
M.J. van~der Laan.
\newblock Efficient estimation in the bivariate censoring model and repairing
  npmle.
\newblock \emph{The Annals of Statistics}, 24\penalty0 (2):\penalty0 596--627,
  1996.

\bibitem[van~der Laan and Robins(2003)]{vanderlaan2003}
M.J. van~der Laan and J.M. Robins.
\newblock \emph{Unified methods for censored longitudinal data and causality}.
\newblock Springer, 2003.

\bibitem[van~der Laan and Rose(2011)]{vanderlaan2011}
M.J. van~der Laan and S.~Rose.
\newblock \emph{Targeted learning: causal inference for observational and
  experimental data}.
\newblock Springer, 2011.

\bibitem[van~der Laan et~al.(2004)van~der Laan, Dudoit, and
  Keles]{vanderlaan2004sagmb}
M.J. van~der Laan, S.~Dudoit, and S.~Keles.
\newblock Asymptotic optimality of likelihood-based cross-validation.
\newblock \emph{Statistical Applications in Genetics and Molecular Biology},
  3\penalty0 (1):\penalty0 1--23, 2004.

\bibitem[van~der Vaart(1991)]{vandervaart1991annals}
A.W. van~der Vaart.
\newblock On differentiable functionals.
\newblock \emph{The Annals of Statistics}, pages 178--204, 1991.

\end{thebibliography}

}
\section*{Appendix}\vspace{.1in}

\begin{proof}[Proof of Theorem 2.]

If condition (a) holds, then the result is true because $\phi^*_{\epsilon,\lambda}$ is a score. We therefore consider the case where it does not hold. Since $\phi^*_{\epsilon,\lambda}\in T_\mathscr{M}(P^*_{\epsilon,\lambda})$, there exists a sequence of one-dimensional regular parametric submodels $\mathscr{M}_{0,m}:=\{P_{\gamma,m}:\gamma\in\Gamma_m\}\subset \mathscr{M}$ with $\Gamma_m\subset\mathbb{R}$ an interval containing zero and with score $s_m$ for $\gamma$ at $\gamma=0$, $m=1,2,\ldots$, such that \[\|\phi^*_{\epsilon,\lambda}-s_m\|_{2,P^*_{\epsilon,\lambda}}\rightarrow 0\] as $m$ tends to infinity. For each $m=1,2,\ldots$, we have that $\int s_m(u)dP_{\epsilon,\lambda}(u)=0$. Because we can write \begin{align*}
\left|\int \phi^*_{\epsilon,\lambda}(u)dP_{\epsilon,\lambda}(u)\right|\ =\ \left|\int \{\phi^*_{\epsilon,\lambda}(u)-s_m(u)\}dP_{\epsilon,\lambda}(u)\right|\ &=\ \left|\int \frac{dP_{\epsilon,\lambda}}{dP^*_{\epsilon,\lambda}}(u)\{\phi^*_{\epsilon,\lambda}(u)-s_m(u)\}dP^*_{\epsilon,\lambda}(u)\right|\\
&\leq\ \|\phi^*_{\epsilon,\lambda}-s_m\|_{2,P^*_{\epsilon,\lambda}}\left\|\frac{dP_{\epsilon,\lambda}}{dP^*_{\epsilon,\lambda}}\right\|_{2,P^*_{\epsilon,\lambda}}
\end{align*} and under condition (b), there exists some $B\in(0,+\infty)$ such that $\|dP_{\epsilon,\lambda}/dP^*_{\epsilon,\lambda}\|_{2,P^*_{\epsilon,\lambda}}<B$ for sufficiently small $\epsilon$ and sufficiently smaller $\lambda$, it must be the case that $\int \phi^*_{\epsilon,\lambda}(u)dP_{\epsilon,\lambda}(u)=0$.

\end{proof}

\begin{proof}[Proof of Theorem 3.]

We first note that \begin{align*}
&\left|\int \phi^*_{\epsilon,\lambda}(u)d(H_{x,\lambda}-P)(u)-\phi_P(x)\right|\\
&\hspace{0.25in}=\ \left|\int \{\phi^*_{\epsilon,\lambda}(u)-\phi_P(u)\}dH_{x,\lambda}(u)+\int\phi_P(u)dH_{x,\lambda}(u)-\phi_P(x)-\int \phi^*_{\epsilon,\lambda}(u)dP(u)\right|\\
&\hspace{.25in}\leq\ \left|\int \{\phi^*_{\epsilon,\lambda}(u)-\phi_P(u)\}dH_{x,\lambda}(u)\right|+\left|\int\phi_P(u)dH_{x,\lambda}(u)-\phi_P(x)\right|+\left|\int \phi^*_{\epsilon,\lambda}(u)dP(u)\right|
\end{align*} and because by assumption the second and third summands on the second line tend to zero as $\lambda$ tends to zero, it suffices to study the first summand. We can bound this term by $\|\phi^*_{\epsilon,\lambda}-\phi_P\|_{\infty,\mathscr{S}_{x,\lambda}}$ and so, if condition (a) holds, the result follows immediately. Alternatively, we can write this term as \begin{align*}
\left|\int \{\phi^*_{\epsilon,\lambda}(u)-\phi_P(u)\}dH_{x,\lambda}(u)\right|\ &=\ \left|\int \frac{dH_{x,\lambda}}{dP}(u)\{\phi^*_{\epsilon,\lambda}(u)-\phi_P(u)\}dP(u)\right|\\
&\leq\ \|\phi^*_{\epsilon,\lambda}-\phi_P\|_{2,P}\left\|\frac{dH_{x,\lambda}}{dP}\right\|_{2,P}
\end{align*} and thus, if condition (b) holds, the result is also guaranteed to hold.

\end{proof}

\begin{proof}[Proof of Theorem 4.]

Using that $P^*_{\epsilon,\lambda}$ is the maximizer of $Q\mapsto \int \log\left[\frac{dQ}{d\nu}(u)\right]dP_{\epsilon,\lambda}(u)$ over all $Q\in\mathscr{M}$, we note that \begin{align*}
0\ \geq\ \int \log\left[\frac{dP^*_{\epsilon,\lambda}}{dP}(u)\right]dP(u)\ &=\ \int \log\left[\frac{dP^*_{\epsilon,\lambda}}{dP}(u)\right]d(P-P_{\epsilon,\lambda})(u)+\int \log\left[\frac{dP^*_{\epsilon,\lambda}}{dP}(u)\right]dP_{\epsilon,\lambda}(u)\\
&\geq\ \int \log\left[\frac{dP^*_{\epsilon,\lambda}}{dP}(u)\right]d(P-P_{\epsilon,\lambda})(u)\\
&=\ \epsilon\int \log\left[\frac{dP^*_{\epsilon,\lambda}}{dP}(u)\right]d(P-H_{x,\lambda})(u)\\
&=\ \epsilon \int \log\left[\frac{dP^*_{\epsilon,\lambda}}{dP}(u)\right]\left\{1-\frac{dH_{x,\lambda}}{dP}(u)\right\}dP(u)\ .
\end{align*} Denoting for any pair $P_1\ll P_2$ the function $u\mapsto\log\left[\frac{dP_1}{dP_2}(u)\right]$ by $L(P_1,P_2)$, this implies that \begin{align*}
\left|\int L(P^*_{\epsilon,\lambda},P)(u)dP(u)\right|\ &\leq\ \epsilon\left|\int L(P^*_{\epsilon,\lambda},P)(u)\left\{1-\frac{dH_{x,\lambda}}{dP}(u)\right\}dP(u)\right|\\
&\leq\ \epsilon \left\|L(P^*_{\epsilon,\lambda},P)\right\|_{2,P}\left\|1-\frac{dH_{x,\lambda}}{dP}\right\|_{2,P}\ \leq\ \epsilon\left\{1+r(\lambda)\right\}\left\|L(P^*_{\epsilon,\lambda},P)\right\|_{2,P}\ .
\end{align*} Provided $P_1\ll P_2$, we have that $\int \{L(P_1,P_2)(u)\}^2dP_2(u)\leq M |\int L(P_1,P_2)(u) dP_2(u)|$, where $M:=M(P_1,P_2)$ depends on the supremum of the Radon-Nikodym of $P_1$ relative to $P_2$ (see, e.g., \citealp{vanderlaan2004sagmb}). This allows us to write that \[\left|\int L(P^*_{\epsilon,\lambda},P)(u)dP(u)\right|\ \leq\  M \epsilon \left\{1+r(\lambda)\right\}\left|\int L(P^*_{\epsilon,\lambda},P)(u)dP(u)\right|^{\frac{1}{2}},\] which directly implies that $|\int L(P^*_{\epsilon,\lambda})(u)dP(u)|\leq M^2\epsilon^2\{1+r(\lambda)\}^2$. If the Radon-Nikodym derivative of $P_1$ relative to $P_2$ is bounded above by $0<\zeta<+\infty$ over the support of $P_2$, we can write that \begin{align*}
0\ \leq\ \left\|\frac{dP_1}{dP_2}-1\right\|^2_{2,P_2}\ &=\ \int \left\{\frac{dP_1}{dP_2}(u)-1\right\}^2dP_2(u)\\
&=\ \int \left\{\sqrt{\frac{dP_1}{dP_2}(u)}-1\right\}^2\left\{\sqrt{\frac{dP_1}{dP_2}(u)}+1\right\}^2dP_2(u)\\
&\leq\ K(\zeta) \int \left\{\sqrt{\frac{dP_1}{dP_2}(u)}-1\right\}^2dP_2(u)\ \leq\ K(\zeta)\left|\int L(P_1,P_2)(u)dP_2(u)\right| ,
\end{align*} where $K(\zeta):=\zeta+2\sqrt{\zeta}+1$ and the last inequality is established using that $\log(u)\leq 2(\sqrt{u}-1)$ for each $u>0$. Thus, by assumption, we find that \begin{align*}
R(P^*_{\epsilon,\lambda},P)\ \leq\ C\left\|\frac{dP^*_{\epsilon,\lambda}}{dP}-1\right\|^2_{2,P}\ \leq\ CK\left|\int L(P^*_{\epsilon,\lambda},P)(u)dP(u)\right|\ \leq\ CKM^2 \epsilon^2\left\{1+r(\lambda)\right\}^2
\end{align*} for small $\lambda$ and sufficiently smaller $\epsilon$, which directly establishes the theorem.\end{proof}

\end{document}